\documentclass[12pt,fleqn]{article}

\usepackage{amsmath,amssymb,amsthm}
\usepackage{hyperref}

\newcommand{\A}{{\mathcal A}}
\newcommand{\abs}[1]{\left|#1\right|}
\newcommand{\as}{\ins{as}}
\newcommand{\bgset}[1]{\big\{#1\big\}}
\newcommand{\dint}{\ds{\int}}
\newcommand{\ds}[1]{\displaystyle #1}
\newcommand{\half}{\frac{1}{2}}
\newcommand{\hquad}{\hspace{0.08in}}
\newcommand{\ins}[1]{\hquad \text{#1} \hquad}
\newcommand{\ip}[3][]{\left(#2,#3\right)_{#1}}
\newcommand{\M}{{\cal M}}
\newcommand{\N}{\mathbb N}
\newcommand{\norm}[2][]{\left\|#2\right\|_{#1}}
\renewcommand{\O}{\text{O}}
\renewcommand{\o}{\text{o}}
\newcommand{\pnorm}[2][]{\if #1'' \left|#2\right|_p \else \left|#2\right|_{#1} \fi}
\newcommand\pow{{(p-2)/p}}
\newcommand{\QED}{\mbox{\qedhere}}
\newcommand{\R}{\mathbb R}
\newcommand{\restr}[2]{\left.#1\right|_{#2}}
\newcommand{\set}[1]{\left\{#1\right\}}
\newcommand\wop{{p/(p-2)}}

\DeclareMathOperator{\sign}{sign}
\DeclareMathOperator{\supp}{supp}

\newenvironment{enumroman}{\begin{enumerate}

}{\end{enumerate}}

\newtheorem{corollary}{Corollary}[section]
\newtheorem{lemma}[corollary]{Lemma}
\newtheorem{proposition}[corollary]{Proposition}
\newtheorem{theorem}[corollary]{Theorem}

\theoremstyle{remark}
\newtheorem{remark}[corollary]{Remark}

\numberwithin{equation}{section}

\title{\vspace{-0.35in} \bf On the Second Minimax Level of the Scalar Field Equation and Symmetry Breaking\thanks{Research supported by funds from the Swedish Research Council.
\newline \smallskip \indent\; {\em MSC2010:} Primary 35J61, 35J20, Secondary 47J10
\newline \smallskip \indent\; {\em Key Words and Phrases:} Scalar field equation, minimax methods, concentration compactness, symmetry breaking}}
\author{\bf Kanishka Perera\thanks{This work was completed while the first-named author was visiting the Department of Mathematics at the Uppsala University, and he is grateful for the kind hospitality of the department.}\\
Department of Mathematical Sciences\\
Florida Institute of Technology\\
Melbourne, FL 32901, USA\\
\em kperera@fit.edu\\
[\bigskipamount]
\bf Cyril Tintarev\\
Department of Mathematics\\
Uppsala University\\
75 106, Uppsala, Sweden\\
\em tintarev@math.uu.se}
\date{}

\begin{document}

\maketitle

\begin{abstract}
We study the second minimax level $\lambda_2$ of the eigenvalue problem for the scalar field equation in $\R^N$. We prove the existence of an eigenfunction at the level $\lambda_2$ when the potential near infinity approaches the constant level from below no faster than $e^{- \varepsilon\, |x|}$. We also consider questions about the nodality of eigenfunctions at this level and establish symmetry breaking at the levels $2,\dots,N$.
\end{abstract}

\section{Introduction}

Consider the eigenvalue problem for the scalar field equation
\begin{equation} \label{1.1}
- \Delta u + V(x)\, u = \lambda\, |u|^{p-2}\, u, \quad x \in \R^N,
\end{equation}
where $N \ge 2$, $V \in L^\infty(\R^N)$ satisfies
\begin{equation} \label{1.2}
\lim_{|x| \to \infty} V(x) = V^\infty > 0,
\end{equation}
$p \in (2,2^\ast)$, and $2^\ast = 2N/(N-2)$ if $N \ge 3$ and $2^\ast = \infty$ if $N = 2$. Let
\[
I(u) = \int_{\R^N} |u|^p, \quad J(u) = \int_{\R^N} |\nabla u|^2 + V(x)\, u^2, \quad u \in H^1(\R^N).
\]
Then the eigenfunctions of \eqref{1.1} on the manifold
\[
\M = \bgset{u \in H^1(\R^N) : I(u) = 1}
\]
and the corresponding eigenvalues coincide with the critical points and the critical values of the constrained functional $\restr{J}{\M}$, respectively.

This problem has been studied extensively for more than three decades. Least energy solutions, also called ground states, are well-understood. In general, the infimum
\begin{equation} \label{1.3}
\lambda_1 := \inf_{u \in \M}\, J(u)
\end{equation}
is not attained. For the autonomous problem at infinity,
\begin{equation} \label{1.4}
- \Delta u + V^\infty\, u = \lambda\, |u|^{p-2}\, u, \quad x \in \R^N,
\end{equation}
the corresponding functional
\[
J^\infty(u) = \int_{\R^N} |\nabla u|^2 + V^\infty\, u^2
\]
attains its infimum
\begin{equation} \label{1.5}
\lambda_1^\infty := \inf_{u \in \M}\, J^\infty(u) > 0
\end{equation}
at a radial function $w^\infty_1 > 0$ and this minimizer is unique up to translations (see Berestycki and Lions \cite{MR695535} and Kwong \cite{MR969899}). For the nonautonomous problem, we have $\lambda_1 \le \lambda_1^\infty$ by the translation invariance of $J^\infty$, and $\lambda_1$ is attained if the inequality is strict (see Lions \cite{MR778970}).

As for the higher energy solutions, also called bound states, radial solutions have been extensively studied when the potential $V$ is radially symmetric (see, e.g., Berestycki and Lions \cite{MR695536}, Grillakis \cite{MR1054554}, and Jones and K{\"u}pper \cite{MR846391}). The subspace $H^1_r(\R^N)$ of $H^1(\R^N)$ consisting of radially symmetric functions is compactly imbedded into $L^p(\R^N)$ for $p \in (2,2^\ast)$ by the compactness result of Strauss \cite{MR0454365}. Denoting by $\Gamma_{m,\, r}$ the class of all odd continuous maps from the unit sphere $S^{m-1} = \bgset{y \in \R^m : |y| = 1}$ to $\M_r = \M \cap H^1_r(\R^N)$, increasing and unbounded sequences of critical values of $\restr{J}{\M_r}$ and $\restr{J^\infty}{\M_r}$ can therefore be defined by
\begin{multline} \label{1.6T}
\lambda_{m,\, r} := \inf_{\gamma \in \Gamma_{m,\, r}}\, \max_{u \in \gamma(S^{m-1})}\, J(u), \quad \lambda_{m,\, r}^\infty := \inf_{\gamma \in \Gamma_{m,\, r}}\, \max_{u \in \gamma(S^{m-1})}\, J^\infty(u),\\
m \in \N,
\end{multline}
respectively. Furthermore, Sobolev imbeddings remain compact for subspaces with any sufficiently robust symmetry (see Bartsch and Wang \cite{MR1349229} and Devillanova and Solimini \cite{MR2895950}). Clapp and Weth \cite{MR2103845} have obtained multiple solutions without any symmetry assumptions, with the number of solutions depending on $N$, under a robust penalty condition similar to \eqref{1.11T} below, but their result does not locate the solutions on particular minimax levels. There is also an extensive literature on multiple solutions of scalar field equations in topologically nontrivial unbounded domains, for which we refer the reader to the survey paper of Cerami \cite{MR2278729}.

In the present paper we study the second minimax levels
\[
\lambda_2 := \inf_{\gamma \in \Gamma_2}\, \max_{u \in \gamma(S^1)}\, J(u), \quad \lambda_2^\infty := \inf_{\gamma \in \Gamma_2}\, \max_{u \in \gamma(S^1)}\, J^\infty(u),
\]
where $\Gamma_2$ is the class of all odd continuous maps from $S^1 = \bgset{y \in \R^2 : |y| = 1}$ to $\M$. It is known that
\begin{equation} \label{1.6}
\lambda_2^\infty = 2^\pow\, \lambda_1^\infty
\end{equation}
is not critical (see, e.g., Weth \cite{MR2263672}). First we give sufficient conditions for $\lambda_2$ to be critical. Recall that
\begin{equation} \label{1.8}
w_1^\infty(x) \sim C_0\, \frac{e^{- \sqrt{V^\infty}\, |x|}}{|x|^{(N-1)/2}} \as |x| \to \infty
\end{equation}
for some constant $C_0 > 0$, and that there are constants $0 < a_0 \le \sqrt{V^\infty}$ and $C > 0$ such that if $\lambda_1$ is attained at $w_1 \ge 0$, then
\begin{equation} \label{1.9}
w_1(x) \le C\, e^{- a_0\, |x|} \quad \forall x \in \R^N
\end{equation}
(see Gidas et al. \cite{MR544879} for the case of constant $V$; the general case follows from an elementary comparison argument). Write
\[
V(x) = V^\infty - W(x),
\]
so that $W(x) \to 0$ as $|x| \to \infty$ by \eqref{1.2}, and write $\pnorm{\cdot}$ for the $L^p$-norm. Our main existence result for the nonautonomous problem \eqref{1.1} is the following.

\begin{theorem} \label{Theorem 1.1}
Assume that $V \in L^\infty(\R^N)$ satisfies \eqref{1.2}, $p \in (2,2^\ast)$, and
\begin{equation} \label{1.11T}
W(x) \ge c\, e^{- a\, |x|} \quad \forall x \in \R^N
\end{equation}
for some constants $0 < a < a_0$ and $c > 0$.
\begin{enumroman}
\item If $\lambda_1 > 0$ and $W \in L^\wop(\R^N)$ with
    \begin{equation} \label{1.10}
    \pnorm[\wop]{W} < \left(2^\pow - 1\right) \lambda_1^\infty,
    \end{equation}
    then the equation \eqref{1.1} has a solution on $\M$ for $\lambda = \lambda_2$.
\item If $\lambda_1 \le 0$, then the equation \eqref{1.1} has a solution on $\M$ for $\lambda = \lambda_2$, and this solution is nodal if $\lambda_1 \le 0 < \lambda_2$ or $\lambda_1 < 0 \le \lambda_2$.
\end{enumroman}
\end{theorem}

The existence of a ground state was initially proved under the penalty condition $V(x) < V^\infty$ by Lions \cite{MR778970}, but it was later relaxed by a term of the order $e^{- a\, |x|}$ by Bahri and Lions \cite{MR1450954}. This can be understood in the sense that the existence of the ground state in the autonomous case is somewhat robust. In our case the same order of correction is involved with the reverse sign, namely, while $\lambda_2^\infty$ is not a critical level for $\restr{J^\infty}{\M}$, it requires the enhanced penalty $V(x) \le V^\infty - c\, e^{- a\, |x|}$ to assure that $\lambda_2$ is critical for $\restr{J}{\M}$. We believe that careful calculations will show that this correction cannot be removed, which in turn suggests that the nonexistence of the second eigenfunction in the autonomous case is equally robust as the existence of the first eigenfunction.

Next we consider the higher minimax levels
\[
\lambda_m := \inf_{\gamma \in \Gamma_m}\, \max_{u \in \gamma(S^{m-1})}\, J(u), \quad \lambda_m^\infty := \inf_{\gamma \in \Gamma_m}\, \max_{u \in \gamma(S^{m-1})}\, J^\infty(u), \quad m \ge 3,
\]
where $\Gamma_m$ is the class of all odd continuous maps from $S^{m-1}$ to $\M$. By \eqref{1.2} and the translation invariance of $J^\infty$,
\begin{equation} \label{1.12}
\lambda_m \le \lambda_m^\infty \quad \forall m \in \N.
\end{equation}
In general, $\lambda_m^\infty$ may be different from the more standard minimax values
\[
\widetilde{\lambda}_m^\infty := \inf_{A \in \A_m}\, \sup_{u \in A}\, J^\infty(u), \quad m \in \N,
\]
where $\A_m$ is the family of all nonempty closed symmetric subsets $A \subset \M$ with genus
\[
i(A) := \inf\, \bgset{n \ge 1 : \exists \ins{an odd continuous map} A \to S^{n-1}} \ge m.
\]
If $\gamma \in \Gamma_m$, then $i(\gamma(S^{m-1})) \ge i(S^{m-1}) = m$ and hence $\gamma(S^{m-1}) \in \A_m$, so
\begin{equation} \label{1.13}
\widetilde{\lambda}_m^\infty \le \lambda_m^\infty \quad \forall m \in \N,
\end{equation}
in particular, $\widetilde{\lambda}_1^\infty = \lambda_1^\infty$. We prove the following nonexistence result for the autonomous problem \eqref{1.4}.

\begin{theorem} \label{Theorem 1.2}
If $p \in (2,2^\ast)$, then
\begin{equation} \label{1.14}
\lambda_m^\infty = 2^\pow\, \lambda_1^\infty = \widetilde{\lambda}_m^\infty, \quad m = 2,\dots,N.
\end{equation}
Hence none of them are critical for $\restr{J^\infty}{\M}$.
\end{theorem}

Finally we prove a symmetry breaking result. Recall that the radial minimax levels $\lambda_{m,\, r}$ defined in \eqref{1.6T} are all critical when $V$ is radial. We have $\lambda_m^\infty \le \lambda_{m,\, r}^\infty$ in general, and since $\lambda_{m,\, r}^\infty$ is critical,
\[
\lambda_m^\infty < \lambda_{m,\, r}^\infty, \quad m = 2,\dots,N
\]
by Theorem \ref{Theorem 1.2}, in particular, $\lambda_{2,\, r}^\infty > \lambda_2^\infty$.

\begin{theorem} \label{Theorem 1.3}
Assume that $V \in L^\infty(\R^N)$ is radial and satisfies \eqref{1.2}, $p \in (2,2^\ast)$, and $W \in L^\wop(\R^N)$ with
\begin{equation} \label{1.15}
\pnorm[\wop]{W} < \lambda_{2,\, r}^\infty - \lambda_2^\infty.
\end{equation}
Then, for $m = 2,\dots,N$,
\begin{enumroman}
\item \label{Theorem 1.3.i} $\lambda_m < \lambda_{m,\, r}$,
\item \label{Theorem 1.3.ii} any solution of \eqref{1.1} on $\M$ with $\lambda = \lambda_m \ge 0$ and $m$ nodal domains is nonradial.
\end{enumroman}
In particular, any nodal solution of \eqref{1.1} on $\M$ with $\lambda = \lambda_2 \ge 0$ is nonradial.
\end{theorem}

\section{Preliminaries}

We will use the norm $\norm{\cdot}$ on $H^1(\R^N)$ induced by the inner product
\[
\ip{u}{v} = \int_{\R^N} \nabla u \cdot \nabla v + V^\infty\, u\, v,
\]
which is equivalent to the standard norm.

\begin{lemma} \label{Lemma 2.1}
Every path $\gamma \in \Gamma_2$ contains a point $u_0$ such that $I(u_0^+) = I(u_0^-)$.
\end{lemma}

\begin{proof}
The function
\[
f(\theta) = I(\gamma(e^{i \theta})^+) - I(\gamma(e^{i \theta})^-), \quad \theta \in [0,\pi]
\]
is continuous since the mappings $u \mapsto u^\pm$ on $H^1(\R^N)$ and the imbedding $H^1(\R^N) \hookrightarrow L^p(\R^N)$ are continuous, $f(\pi) = - f(0)$ since $\gamma$ is odd and $(-u)^\pm = u^\mp$, so $f(\theta) = 0$ for some $\theta \in [0,\pi]$ by the intermediate value theorem.
\end{proof}

For $u_1, u_2 \in \M$ with $u_2 \ne \pm u_1$, consider the path in $\Gamma_2$ given by
\[
\gamma_{u_1 u_2}(e^{i \theta}) = \frac{u_1\, \cos \theta + u_2\, \sin \theta}{\pnorm{u_1\, \cos \theta + u_2\, \sin \theta}}, \quad \theta \in [0,2 \pi],
\]
which passes through $u_1$ and $u_2$.

\begin{lemma} \label{Lemma 2.2}
If $u_1, u_2 \in \M$ have disjoint supports, then
\[
\max_{u \in \gamma_{u_1 u_2}}\, J(u) = \begin{cases}
\left(J(u_1)^\wop + J(u_2)^\wop\right)^\pow, & \hspace{-8pt} J(u_1), J(u_2) > 0\\[10pt]
J(u_2), & \hspace{-8pt} J(u_1) \le 0 < J(u_2)\\[10pt]
- \left(|J(u_1)|^\wop + |J(u_2)|^\wop\right)^\pow, & \hspace{-8pt} J(u_1), J(u_2) \le 0.
\end{cases}
\]
\end{lemma}

\begin{proof}
We have
\[
J(\gamma_{u_1 u_2}(e^{i \theta})) = \frac{J(u_1)\, \cos^2 \theta + J(u_2)\, \sin^2 \theta}{\big(|\cos \theta|^p + |\sin \theta|^p\big)^{2/p}},
\]
and a straightforward calculation yields the conclusion.
\end{proof}

For each $y \in \R^N$, the translation $u \mapsto u(\cdot - y)$ is a unitary operator on $H^1(\R^N)$ and an isometry of $L^p(\R^N)$, in particular, it preserves $\M$.

\begin{lemma} \label{Lemma 2.3}
Given $u_1, u_2 \in \M$ and $\varepsilon > 0$, there is a path $\gamma \in \Gamma_2$ such that
\[
\max_{u \in \gamma}\, J(u) < \begin{cases}
\left(J(u_1)^\wop + J^\infty(u_2)^\wop\right)^\pow + \varepsilon, & J(u_1) > 0\\[10pt]
J^\infty(u_2) + \varepsilon, & J(u_1) \le 0.
\end{cases}
\]
\end{lemma}

\begin{proof}
By density, $u_1$ and $u_2$ can be approximated by functions $\widetilde{u}_1, \widetilde{u}_2 \in C^\infty_0(\R^N) \cap \M$, respectively. For all $y \in \R^N$ with $|y|$ sufficiently large, $\widetilde{u}_1$ and $\widetilde{u}_2(\cdot - y)$ have disjoint supports, and, by \eqref{1.2} and the dominated convergence theorem,
\[
\lim_{|y| \to \infty} J(\widetilde{u}_2(\cdot - y)) = J^\infty(\widetilde{u}_2) > 0,
\]
so the conclusion follows from Lemma \ref{Lemma 2.2} and the continuity of $J$ and $J^\infty$.
\end{proof}

We can now obtain some bounds for $\lambda_2$.

\begin{proposition} \label{Proposition 1.1}
Assume that $V \in L^\infty(\R^N)$ satisfies \eqref{1.2} and $p \in (2,2^\ast)$.
\begin{enumroman}
\item If $\lambda_1 > 0$, then
\[
2^\pow\, \lambda_1 \le \lambda_2 \le \big(\lambda_1^\wop + (\lambda_1^\infty)^\wop\big)^\pow.
\]
In particular, $\lambda_2^\infty = 2^\pow\, \lambda_1^\infty$.
\item If $\lambda_1 \le 0$, then
\[
\lambda_1 \le \lambda_2 \le \lambda_1^\infty.
\]
\end{enumroman}
\end{proposition}

\begin{proof}
Every path $\gamma \in \Gamma_2$ contains a point $u_0$ with $\pnorm{u_0^\pm} = 1/2^{1/p}$ by Lemma \ref{Lemma 2.1}, and hence
\[
\max_{u \in \gamma}\, J(u) \ge J(u_0) = J(u_0^+) + J(u_0^-) \ge \lambda_1 \pnorm{u_0^+}^2 + \lambda_1 \pnorm{u_0^-}^2 = 2^\pow\, \lambda_1,
\]
so $\lambda_2 \ge 2^\pow\, \lambda_1$. This proves the lower bounds for $\lambda_2$ since $\lambda_2 \ge \lambda_1$ in general. The upper bounds follow by applying Lemma \ref{Lemma 2.3} to minimizing sequences for $\lambda_1$ and $\lambda_1^\infty$.
\end{proof}

For $u \in H^1(\R^N) \setminus \set{0}$, denote by $\widehat{u} = u/\pnorm{u}$ the radial projection of $u$ on $\M$. Recall that $u$ is called nodal if both $u^+$ and $u^-$ are nonzero.

\begin{lemma} \label{Lemma 2.4}
If $u_0 \in \M$ is a nodal critical point of $J$, then
\[
\max_{u \in \gamma_{\widehat{u_0^+} \widehat{u_0^-}}}\, J(u) = J(u_0),
\]
and hence $\lambda_2 \le J(u_0)$.
\end{lemma}

\begin{proof}
Testing the equation \eqref{1.1} for $u = u_0$ with $u_0, u_0^\pm$ gives
\[
J(u_0) = \lambda, \qquad J(u_0^\pm) = \lambda \pnorm{u_0^\pm}^p,
\]
respectively, so
\[
J(\widehat{u_0^\pm}) = \frac{J(u_0^\pm)}{\pnorm{u_0^\pm}^2} = J(u_0) \pnorm{u_0^\pm}^{p-2},
\]
in particular, $J(\widehat{u_0^\pm})$ have the same sign as $J(u_0)$. Since $\widehat{u_0^\pm}$ have disjoint supports, then
\[
\max_{u \in \gamma_{\widehat{u_0^+} \widehat{u_0^-}}}\, J(u) = \sign(J(u_0))\, \big(|J(\widehat{u_0^+})|^\wop + |J(\widehat{u_0^-})|^\wop\big)^\pow
\]
by Lemma \ref{Lemma 2.2}, and the conclusion follows since $\pnorm{u_0^+}^p + \pnorm{u_0^-}^p = 1$.
\end{proof}

We can now prove the following nonexistence results. This proposition is well-known, but we include a short proof here for the convenience of the reader.

\begin{proposition} \label{Proposition 1.2}
Assume that $V \in L^\infty(\R^N)$ satisfies \eqref{1.2} and $p \in (2,2^\ast)$.
\begin{enumroman}
\item \label{1.2.i} The equation \eqref{1.1} has no nodal solution on $\M$ for $\lambda < \lambda_2$.
\item \label{1.2.ii} The equation \eqref{1.4} has no solution on $\M$ for $\lambda_1^\infty < \lambda \le \lambda_2^\infty$.
\end{enumroman}
\end{proposition}

\begin{proof}
Part \ref{1.2.i} is immediate from Lemma \ref{Lemma 2.4}. As for part \ref{1.2.ii}, there is no solution for $\lambda_1^\infty < \lambda < \lambda_2^\infty$ by part \ref{1.2.i} since $\pm w_1^\infty$ are the only sign definite solutions (see Kwong \cite{MR969899}). If $u_0 \in \M$ is a solution of \eqref{1.4} with $\lambda = \lambda_2^\infty$, then $u_0$ is nodal, so as in the proof of Lemma \ref{Lemma 2.4},
\[
\begin{split}
\lambda_2^\infty = J^\infty(u_0) = \left(J^\infty(\widehat{u_0^+})^\wop + J^\infty(\widehat{u_0^-})^\wop\right)^\pow\\
\ge \left((\lambda_1^\infty)^\wop + (\lambda_1^\infty)^\wop\right)^\pow = 2^\pow\, \lambda_1^\infty.
\end{split}
\]
By \eqref{1.6}, equality holds throughout and $\widehat{u_0^\pm}$ are minimizers for $\lambda_1^\infty$. Since any nonnodal solution is sign definite by the strong maximum principle, then both $u_0^+$ and $u_0^-$ are positive everywhere, a contradiction.
\end{proof}

Recall that a nodal domain of $u \in H^1(\R^N)$ is a nonempty connected component of $\R^N \setminus u^{-1}(\set{0})$.

\begin{lemma} \label{Lemma 2.5}
If $u_0 \in \M$ is a critical point of $J$ with $m$ (or more) nodal domains and $J(u_0) \ge 0$, then there is a map $\gamma \in \Gamma_m$ such that
\[
\max_{u \in \gamma(S^{m-1})}\, J(u) \le J(u_0),
\]
and hence $\lambda_m \le J(u_0)$. If, in addition, $V$ is radial and $u_0 \in \M_r$, then $\gamma \in \Gamma_{m,\, r}$, so $\lambda_{m,\, r} \le J(u_0)$.
\end{lemma}

\begin{proof}
Let $\Omega_j,\, j = 1,\dots,m$ be distinct nodal domains of $u_0$, and set $u_j = \chi_{\Omega_j}\, u_0$, where $\chi_{\Omega_j}$ is the characteristic function of $\Omega_j$. Then $u_j \in H^1(\R^N)$ have pairwise disjoint supports. Define $\gamma \in \Gamma_m$ by
\[
\gamma(y) = \widehat{\sum_{j=1}^m\, y_j\, \widehat{u}_j}, \quad y = (y_1,\dots,y_m) \in S^{m-1}.
\]
Testing the equation \eqref{1.1} for $u = u_0$ with $u_0, u_j$ gives
\[
J(u_0) = \lambda, \qquad J(u_j) = \lambda \pnorm{u_j}^p,
\]
respectively, so
\[
J(\widehat{u}_j) = \frac{J(u_j)}{\pnorm{u_j}^2} = J(u_0) \pnorm{u_j}^{p-2}.
\]
Thus,
\begin{multline*}
J(\gamma(y)) = \frac{\sum_{j=1}^m\, y_j^2\, J(\widehat{u}_j)}{\left(\sum_{j=1}^m\, |y_j|^p \pnorm{\widehat{u}_j}^p\right)^{2/p}} = J(u_0)\, \frac{\sum_{j=1}^m\, y_j^2 \pnorm{u_j}^{p-2}}{\left(\sum_{j=1}^m\, |y_j|^p\right)^{2/p}}\\[5pt]
\le J(u_0) \left(\sum_{j=1}^m\, \pnorm{u_j}^p\right)^\pow \le J(u_0) \pnorm{u_0}^{p-2} = J(u_0)
\end{multline*}
by the H\"older inequality for sums. If $u_0 \in \M_r$, then each $u_j \in H^1_r(\R^N)$ and hence $\gamma \in \Gamma_{m,\, r}$.
\end{proof}

Recall that $u_k \in \M$ is a critical sequence for $\restr{J}{\M}$ at the level $c \in \R$ if
\begin{equation} \label{3.1}
J'(u_k) - \mu_k\, I'(u_k) \to 0, \qquad J(u_k) \to c
\end{equation}
for some sequence $\mu_k \in \R$. Since $\ip{J'(u_k)}{u_k} = 2\, J(u_k)$ and $\ip{I'(u_k)}{u_k} = p\, I(u_k) = p$, then $\mu_k \to (2/p)\, c$.

\begin{lemma} \label{Lemma 3.1}
Any sublevel set of $\restr{J}{\M}$ is bounded, and $\lambda_1 > - \infty$. In particular, any critical sequence $u_k$ for $\restr{J}{\M}$ is bounded.
\end{lemma}

\begin{proof}
Let $u_k \in \M$ be any sequence such that $J(u_k) \le \alpha < \infty$. Then
\[
\int_{\R^N} \Big(|\nabla u_k|^2 + \half\, V_\infty\, u_k^2\Big) \le \alpha + \int_{\R^N} \Big(W - \half\, V_\infty\Big)^+\, u_k^2.
\]
Note that the set $D = \supp \left(W - 1/2\, V_\infty\right)^+$ has finite measure and $W \in L^\infty(\R^N)$, so the second term on the right hand side, by the H\"older inequality on $D$, is bounded by $C\, |u_k|_p^2$ and hence bounded. Finally, it remains to note that $J$ is a sum of $\norm{u}^2$ and a weakly continuous functional, and thus it is weakly lower semicontinuous. Since its sublevel sets are bounded, it is necessarily bounded from below.
\end{proof}

In the absence of a compact Sobolev imbedding, the main technical tool we use here for handling the convergence matters is the concentration compactness principle of Lions \cite{MR778970,MR879032}. This is expressed as the profile decomposition of Benci and Cerami \cite{MR898712} for critical sequences of $\restr{J}{\M}$, which is a particular case of the profile decomposition of Solimini \cite{MR1340267} for general sequences in Sobolev spaces.

\begin{proposition} \label{Proposition 3.2}
Let $u_k \in H^1(\R^N)$ be a bounded sequence, and assume that there is a constant $\delta > 0$ such that if $u_k(\cdot + y_k) \rightharpoonup w \ne 0$ on a renumbered subsequence for some $y_k \in \R^N$ with $|y_k| \to \infty$, then $\norm{w} \ge \delta$. Then there are $m \in \N$, $w^{(n)} \in H^1(\R^N)$, $y^{(n)}_k \in \R^N,\, y^{(1)}_k = 0$ with $k \in \N$, $n \in \set{1,\dots,m}$, $w^{(n)} \ne 0$ for $n \ge 2$, such that, on a renumbered subsequence,
\begin{gather}
u_k(\cdot + y^{(n)}_k) \rightharpoonup w^{(n)}, \label{3.2}\\[10pt]
\big|y^{(n)}_k - y^{(l)}_k\big| \to \infty \text{ for } n \ne l, \label{3.3}\\[7.5pt]
\sum_{n=1}^m\, \norm{w^{(n)}}^2 \le \liminf \norm{u_k}^2, \label{3.4}\\[5pt]
u_k - \sum_{n=1}^m\, w^{(n)}(\cdot - y^{(n)}_k) \to 0 \text{ in } L^p(\R^N) \quad \forall p \in (2,2^\ast). \label{3.5}
\end{gather}
\end{proposition}

Equation \eqref{3.1} implies
\begin{equation} \label{3.6}
- \Delta u_k + V(x)\, u_k = c_k\, |u_k|^{p-2}\, u_k + \o(1),
\end{equation}
where $c_k = (p/2)\, \mu_k \to c$. So if $u_k(\cdot + y_k) \rightharpoonup w$ on a renumbered subsequence for some $y_k \in \R^N$ with $|y_k| \to \infty$, then $w$ solves \eqref{1.4} with $\lambda = c$ by \eqref{1.2}, in particular, $\norm{w}^2 = c \pnorm{w}^p$. If $w \ne 0$, it follows that $c > 0$ and $\norm{w} \ge \big((\lambda_1^\infty)^p/c\big)^{1/2\, (p-1)}$ since $\norm{w}^2/\pnorm{w}^2 \ge \lambda_1^\infty$.

\begin{proposition} \label{Proposition 3.3}
Let $u_k \in \M$ be a critical sequence for $\restr{J}{\M}$ at the level $c \in \R$. Then it admits a renumbered subsequence that satisfies the conclusions of Proposition \ref{Proposition 3.2} for some $m \in \N$, and, in addition,
\begin{gather}
- \Delta w^{(1)} + V(x)\, w^{(1)} = c\, |w^{(1)}|^{p-2}\, w^{(1)}, \notag\\[5pt]
\qquad - \Delta w^{(n)} + V^\infty\, w^{(n)} = c\, |w^{(n)}|^{p-2}\, w^{(n)}, \quad n = 2,\dots,m, \label{3.7}\\[10pt]
J(w^{(1)}) = c\, I(w^{(1)}), \qquad J^\infty(w^{(n)}) = c\, I(w^{(n)}), \quad n = 2,\dots,m, \label{3.8}\\[7.5pt]
\sum_{n=1}^m\, I(w^{(n)}) = 1, \qquad J(w^{(1)}) + \sum_{n=2}^m\, J^\infty(w^{(n)}) = c, \label{3.9}\\[5pt]
u_k - \sum_{n=1}^m\, w^{(n)}(\cdot - y^{(n)}_k) \to 0 \text{ in } H^1(\R^N). \label{3.10}
\end{gather}
\end{proposition}

\begin{proof}
The proof is based on standard arguments and we only sketch it. Equations in \eqref{3.7} follow from \eqref{3.6}, \eqref{3.2}, and \eqref{1.2}, and \eqref{3.8} is immediate from \eqref{3.7}. First equation in \eqref{3.9} is a particular case of Lemma 3.4 in Tintarev and Fieseler \cite{MR2294665}, and the second follows from \eqref{3.8} and the first. Relation \eqref{3.10} follows from \eqref{3.5}, \eqref{3.6}, and the continuity of the Sobolev imbedding.
\end{proof}

We can now show that $\restr{J}{\M}$ satisfies the Palais-Smale condition in a range of levels strictly below the upper bound given by Proposition \ref{Proposition 1.1},
\[
\lambda^\# = \begin{cases}
\big(\lambda_1^\wop + (\lambda_1^\infty)^\wop\big)^\pow, & \lambda_1 > 0\\[10pt]
\lambda_1^\infty, & \lambda_1 \le 0.
\end{cases}
\]
Note that
\begin{equation} \label{1.7}
\lambda_1^\infty \le \lambda^\# \le \lambda_2^\infty.
\end{equation}
Let $u_k$ be the renumbered subsequence of a critical sequence for $\restr{J}{\M}$ at the level $c$ given by Proposition \ref{Proposition 3.3}, and set $t_n = I(w^{(n)})$. Then
\begin{equation} \label{3.11}
\sum_{n=1}^m\, t_n = 1
\end{equation}
by \eqref{3.9}, so each $t_n \in [0,1]$, and $t_n \ne 0$ for $n \ge 2$. Since $J(w^{(1)}) \ge \lambda_1\, t_1^{2/p}$ and $J^\infty(w^{(n)}) \ge \lambda_1^\infty\, t_n^{2/p}$, \eqref{3.8} gives
\begin{equation} \label{3.12}
t_1 = 0 \ins{or} c\, t_1^\pow \ge \lambda_1, \qquad c\, t_n^\pow \ge \lambda_1^\infty, \quad n = 2,\dots,m.
\end{equation}
It follows from \eqref{3.11} and \eqref{3.12} that if $m \ge 2$, then
\begin{equation} \label{3.13}
c \ge \begin{cases}
\big(\lambda_1^\wop + (m - 1)\, (\lambda_1^\infty)^\wop\big)^\pow, & t_1 \ne 0 \text{ and } \lambda_1 > 0\\[10pt]
(m - 1)^\pow\, \lambda_1^\infty, & t_1 = 0 \text{ or } \lambda_1 \le 0.
\end{cases}
\end{equation}

\begin{lemma} \label{Lemma 3.4}
$u_k$ converges in $H^1(\R^N)$ in the following cases:
\begin{enumroman}
\item \label{3.3.i} $\lambda_1 > 0$ and $\lambda_1^\infty < c < \lambda^\#$,
\item \label{3.3.ii} $\lambda_1 \le 0$ and $c < \lambda_1^\infty = \lambda^\#$.
\end{enumroman}
\end{lemma}

\begin{proof}
If $m = 1$, then $u_k \to w^{(1)}$ in $H^1(\R^N)$ by \eqref{3.10}, so suppose $m \ge 2$. Then
\[
(m - 1)^\pow\, \lambda_1^\infty \le c < \lambda^\# \le \lambda_2^\infty = 2^\pow\, \lambda_1^\infty
\]
by \eqref{3.13}, \eqref{1.7}, and \eqref{1.6}, so $m = 2$ and $c \ge \lambda_1^\infty$, which eliminates case \ref{3.3.ii}. As for case \ref{3.3.i}, if $t_1 \ne 0$, then $c \ge \lambda^\#$ by \eqref{3.13}, so $t_1 = 0$. Then $t_2 = 1$ by \eqref{3.11}, so $w^{(2)}$ is a solution of \eqref{1.4} on $\M$ with $\lambda = c$ by \eqref{3.7}, which contradicts Proposition \ref{Proposition 1.2} since $\lambda_1^\infty < c < \lambda^\# \le \lambda_2^\infty$.
\end{proof}

\begin{remark}
Lemma \ref{Lemma 3.4} is due to Cerami \cite{MR2278729}. We have included a proof here merely for the convenience of the reader.
\end{remark}

We now have the following existence results for \eqref{1.1}.

\begin{proposition} \label{Proposition 1.3}
Assume that $V \in L^\infty(\R^N)$ satisfies \eqref{1.2} and $p \in (2,2^\ast)$. Then the equation \eqref{1.1} has a solution on $\M$ for $\lambda = \lambda_2$ in the following cases:
\begin{enumroman}
\item $\lambda_1 > 0$ and $\lambda_1^\infty < \lambda_2 < \lambda^\#$,
\item $\lambda_1 \le 0$ and $\lambda_2 < \lambda_1^\infty = \lambda^\#$.
\end{enumroman}
\end{proposition}

\begin{proof}
Since $\restr{J}{\M}$ satisfies the Palais-Smale condition at the level $\lambda_2$ by Lemma \ref{Lemma 3.4}, it is a critical level by a standard argument.
\end{proof}

\begin{lemma} \label{Lemma 3.5}
If $\lambda_1 \le 0 < \lambda$ or $\lambda_1 < 0 \le \lambda$, then every solution $u$ of \eqref{1.1} on $\M$ is nodal.
\end{lemma}

\begin{proof}
Since $\lambda_1 \le 0 < \lambda_1^\infty$, \eqref{1.1} has a solution $w_1 > 0$ on $\M$ for $\lambda = \lambda_1$ (see Lions \cite{MR778970}). Then
\begin{equation} \label{3.14}
\lambda_1 \int_{\R^N} w_1^{p-1}\, u = \int_{\R^N} \nabla w_1 \cdot \nabla u + V(x)\, w_1\, u = \lambda \int_{\R^N} |u|^{p-2}\, u\, w_1.
\end{equation}
If $u$ is nonnodal, then it is sign definite by the strong maximum principle, so \eqref{3.14} implies that $\lambda_1$ and $\lambda$ have the same sign.
\end{proof}

\section{Proof of Theorem \ref{Theorem 1.1}}

Theorem \ref{Theorem 1.1} follows from Proposition \ref{Proposition 1.3}, Proposition \ref{Proposition 1.4} below, and Lemma \ref{Lemma 3.5}.

\begin{proposition} \label{Proposition 1.4}
Assume that $V \in L^\infty(\R^N)$ satisfies \eqref{1.2} and $p \in (2,2^\ast)$.
\begin{enumroman}
\item \label{1.4.ii} If \eqref{1.11T} holds, then $\lambda_2 < \lambda^\#$.
\item \label{1.4.i} If $\lambda_1 > 0$ and \eqref{1.10} holds, then $\lambda_2 > \lambda_1^\infty$.
\end{enumroman}
\end{proposition}

Under assumption \eqref{1.11T},
\[
\lambda_1 \le J(w_1^\infty) \le J^\infty(w_1^\infty) - c \int_{\R^N} e^{- a\, |x|}\, w_1^\infty(x)^2\, dx < \lambda_1^\infty,
\]
so $\lambda_1$ is attained at some function $w_1 \ge 0$ (see Lions \cite{MR778970}). Our idea of the proof for part \ref{1.4.ii} of Proposition \ref{Proposition 1.4} is to show, analogously, that if $|y|$ is sufficiently large, then
\[
\lambda_2 \le \max_{u \in \gamma_{w_1 w_1^\infty(\cdot - y)}}\, J(u) < \lambda^\#.
\]

\begin{lemma} \label{Lemma 4.1}
Let $a_0$ be as in \eqref{1.9}. Then, as $|y| \to \infty$,
\begin{enumroman}
\item \label{4.1.i} $\dint_{\R^N} w_1(x)^{p-1}\, w_1^\infty(x - y)\, dx = \O(e^{- a_0\, |y|})$,
\item \label{4.1.ii} $\dint_{\R^N} w_1(x)\, w_1^\infty(x - y)^{p-1}\, dx = \O(e^{- a_0\, |y|})$,
\item \label{4.1.iii} $J(w_1\, \cos \theta + w_1^\infty(\cdot - y)\, \sin \theta)\\[5pt]
    = \lambda_1\, \cos^2 \theta + \left(\lambda_1^\infty - \dint_{\R^N} W(x)\, w_1^\infty(x - y)^2\, dx\right) \sin^2 \theta + \O(e^{- a_0\, |y|})$,
\item \label{4.1.iv} $\pnorm{w_1\, \cos \theta + w_1^\infty(\cdot - y)\, \sin \theta}^2 \ge \big(|\cos \theta|^p + |\sin \theta|^p\big)^{2/p} + \O(e^{- a_0\, |y|})$.
\end{enumroman}
\end{lemma}

\begin{proof}
\ref{4.1.i} By \eqref{1.8}, $w_1^\infty(x) \le \widetilde{C}\, e^{- a_0\, |x|}$ for some $\widetilde{C} > 0$, which together with \eqref{1.9} shows that the integral on the left is bounded by a constant multiple of
\[
\int_{\R^N} e^{- a_0\, [(p - 1)\, |x| + |x - y|]} \le e^{- a_0\, |y|} \int_{\R^N} e^{- a_0\, (p - 2)\, |x|}
\]
by the triangle inequality.

\ref{4.1.ii} Same as part \ref{4.1.i} after the change of variable $x \mapsto x + y$.

\ref{4.1.iii} We have
\begin{align*}
& \hquad J(w_1\, \cos \theta + w_1^\infty(\cdot - y)\, \sin \theta)\\[7.5pt]
= & \hquad J(w_1)\, \cos^2 \theta + \left(J^\infty(w_1^\infty(\cdot - y)) - \int_{\R^N} W(x)\, w_1^\infty(x - y)^2\, dx\right) \sin^2 \theta\\[5pt]
& \hquad + \sin 2 \theta \int_{\R^N} \big(\nabla w_1(x) \cdot \nabla w_1^\infty(x - y) + V(x)\, w_1(x)\, w_1^\infty(x - y)\big)\, dx\\[5pt]
= & \hquad \lambda_1\, \cos^2 \theta + \left(\lambda_1^\infty - \int_{\R^N} W(x)\, w_1^\infty(x - y)^2\, dx\right) \sin^2 \theta\\[5pt]
& \hquad + \lambda_1\, \sin 2 \theta \int_{\R^N} w_1(x)^{p-1}\, w_1^\infty(x - y)\, dx
\end{align*}
since $w_1$ solves \eqref{1.1} with $\lambda = \lambda_1$, and the last term is of the order $\O(e^{- a_0\, |y|})$ by part \ref{4.1.i}.

\ref{4.1.iv} Using the elementary inequality
\[
|a + b|^p \ge |a|^p + |b|^p - p\, |a|^{p-1}\, |b| - p\, |a|\, |b|^{p-1} \quad \forall a, b \in \R,
\]
we have
\begin{align*}
& \hquad \pnorm{w_1\, \cos \theta + w_1^\infty(\cdot - y)\, \sin \theta}^2\\[7.5pt]
\ge & \hquad \bigg(\pnorm{w_1}^p\, |\cos \theta|^p + \pnorm{w_1^\infty(\cdot - y)}^p\, |\sin \theta|^p\\[5pt]
& \hquad - p \int_{\R^N} w_1(x)^{p-1}\, w_1^\infty(x - y)\, dx - p \int_{\R^N} w_1(x)\, w_1^\infty(x - y)^{p-1}\, dx\bigg)^{2/p}\\[7.5pt]
= & \hquad \big(|\cos \theta|^p + |\sin \theta|^p + \O(e^{- a_0\, |y|})\big)^{2/p}
\end{align*}
by parts \ref{4.1.i} and \ref{4.1.ii}, and the conclusion follows.
\end{proof}

\begin{lemma} \label{Lemma 4.2}
If $W \in L^\wop(\R^N)$, then
\[
\sup_{u \in \M}\, |J(u) - J^\infty(u)| \le \pnorm[\wop]{W}.
\]
\end{lemma}

\begin{proof}
For $u \in \M$,
\[
|J(u) - J^\infty(u)| \le \int_{\R^N} |W(x)|\, u^2 \le \pnorm[\wop]{W} \pnorm{u}^2 = \pnorm[\wop]{W}
\]
by the H\"older inequality.
\end{proof}

\begin{proof}[Proof of Proposition \ref{Proposition 1.4}]
\ref{1.4.ii} By \eqref{1.8} and \eqref{1.11T},
\[
\int_{\R^N} W(x)\, w_1^\infty(x - y)^2\, dx \ge \widetilde{c}\, e^{- a\, |y|} \quad \forall y \in \R^N
\]
for some $\widetilde{c} > 0$, which together with Lemma \ref{Lemma 4.1} gives
\begin{align*}
& \hquad J(\gamma_{w_1 w_1^\infty(\cdot - y)}(e^{i \theta}))\\[10pt]
= & \hquad \frac{J(w_1\, \cos \theta + w_1^\infty(\cdot - y)\, \sin \theta)}{\pnorm{w_1\, \cos \theta + w_1^\infty(\cdot - y)\, \sin \theta}^2}\\[10pt]
\le & \hquad \frac{\lambda_1\, \cos^2 \theta + \left(\lambda_1^\infty - \widetilde{c}\, e^{- a\, |y|}\right) \sin^2 \theta}{\big(|\cos \theta|^p + |\sin \theta|^p\big)^{2/p}} + \O(e^{- a_0\, |y|})\\[10pt]
\le & \hquad \begin{cases}
\left[\lambda_1^\wop + \left(\lambda_1^\infty - \widetilde{c}\, e^{- a\, |y|}\right)^\wop\right]^\pow + \O(e^{- a_0\, |y|}), & \lambda_1 > 0\\[12.5pt]
\lambda_1^\infty - \widetilde{c}\, e^{- a\, |y|} + \O(e^{- a_0\, |y|}), & \lambda_1 \le 0
\end{cases}\\[10pt]
< & \hquad \lambda^\#
\end{align*}
if $|y|$ is sufficiently large since $a < a_0$.

\ref{1.4.i} By Lemma \ref{Lemma 4.2}, Proposition \ref{Proposition 1.1}, and \eqref{1.10},
\[
\lambda_2 \ge \lambda_2^\infty - \pnorm[\wop]{W} > \lambda_1^\infty. \QED
\]
\end{proof}

\section{Proofs of Theorems \ref{Theorem 1.2} and \ref{Theorem 1.3}}

\begin{proof}[Proof of Theorem \ref{Theorem 1.2}]
An approximation argument as in the proof of Lemma \ref{Lemma 2.3} shows that given $\varepsilon > 0$, there is a $R > 0$ such that, for $\gamma_R \in \Gamma_N$ given by
\[
\gamma_R(y) = \frac{w_1^\infty(\cdot + Ry) - w_1^\infty(\cdot - Ry)}{\pnorm{w_1^\infty(\cdot + Ry) - w_1^\infty(\cdot - Ry)}}, \quad y \in S^{N-1},
\]
we have
\[
J^\infty(\gamma_R(y)) < 2^\pow\, \lambda_1^\infty + \varepsilon \quad \forall y \in S^{N-1}.
\]
So $\lambda_N^\infty \le 2^\pow\, \lambda_1^\infty$, and the first equality in \eqref{1.14} then follows since, by Proposition \ref{Proposition 1.1}, $2^\pow\, \lambda_1^\infty = \lambda_2^\infty \le \lambda_N^\infty$.

If $A \in \A_2$, then $A$ contains a point $u_0$ with $\pnorm{u_0^\pm} = 1/2^{1/p}$ since otherwise
\[
A \to S^0, \quad u \mapsto \frac{\pnorm{u^+} - \pnorm{u^-}}{\abs{\pnorm{u^+} - \pnorm{u^-}}}
\]
is an odd continuous map and hence $i(A) = 1$, so
\begin{multline*}
\sup_{u \in A}\, J^\infty(u) \ge J^\infty(u_0) = J^\infty(u_0^+) + J^\infty(u_0^-)\\
\ge \lambda_1^\infty \pnorm{u_0^+}^2 + \lambda_1^\infty \pnorm{u_0^-}^2 = 2^\pow\, \lambda_1^\infty = \lambda_2^\infty
\end{multline*}
by Proposition \ref{Proposition 1.1}. So $\widetilde{\lambda}_2^\infty \ge \lambda_2^\infty$, and the second equality in \eqref{1.14} then follows from the first since $\widetilde{\lambda}_2^\infty \le \widetilde{\lambda}_N^\infty \le \lambda_N^\infty$ by \eqref{1.13}.
\end{proof}

\begin{proof}[Proof of Theorem \ref{Theorem 1.3}]
\ref{Theorem 1.3.i} By Lemma \ref{Lemma 4.2}, \eqref{1.15}, Theorem \ref{Theorem 1.2}, and \eqref{1.12},
\[
\lambda_{m,\, r} \ge \lambda_{m,\, r}^\infty - \pnorm[\wop]{W} \ge \lambda_{2,\, r}^\infty - \pnorm[\wop]{W} > \lambda_2^\infty = \lambda_m^\infty \ge \lambda_m.
\]

\ref{Theorem 1.3.ii} If \eqref{1.1} has a solution $u_0 \in \M_r$ for $\lambda = \lambda_m \ge 0$ with $m$ nodal domains, then $\lambda_{m,\, r} \le J(u_0) = \lambda_m$ by Lemma \ref{Lemma 2.5}, contradicting \ref{Theorem 1.3.i}.
\end{proof}

We close with some questions related to the present paper that remain open.
\begin{enumroman}
\item When is every solution of \eqref{1.1} at the level $\lambda_2$ nodal? Theorem \ref{Theorem 1.1} gives only a partial answer. What can be said about the geometry of the nodal domains for nodal solutions at this level?
\item Assuming that $V$ is radial, and taking into account the symmetry analysis of Gladiali et al. \cite{MR2609032}, is every solution at the level $\lambda_2$ foliated Schwarz symmetric?
\item Can the enhanced penalty condition \eqref{1.11T} be relaxed?
\item Is there an analog of Proposition \ref{Proposition 1.3} for higher minimax levels?
\end{enumroman}

\def\cdprime{$''$}

\end{document}